\documentclass[reqno]{amsart}
\usepackage[T1]{fontenc}
\usepackage{lmodern,amsmath,amssymb,amsthm,mathtools}
\usepackage[margin=1in]{geometry}
\usepackage{microtype,graphicx,tikz,xcolor}
\usepackage[hidelinks]{hyperref}
\hypersetup{pdftitle={Long-range expanders: construction and cutoff},pdfauthor={Dylan J. Altschuler}}
\usepackage[font=small,labelfont=bf]{caption}
\newtheorem{theorem}{Theorem}[section]
\newtheorem{proposition}[theorem]{Proposition}
\newtheorem{lemma}[theorem]{Lemma}
\newtheorem{corollary}[theorem]{Corollary}
\newtheorem{conjecture}[theorem]{Conjecture}
\theoremstyle{definition}
\newtheorem{remark}[theorem]{Remark}
\newtheorem{definition}[theorem]{Definition}

\newcommand{\TV}{\mathrm{TV}}
\newcommand{\dist}{\operatorname{dist}}

\newcommand{\mix}{\mathrm{mix}}

\title{Long-range expanders: construction and cutoff}
\author{Dylan J. Altschuler}

\address{Dylan J. Altschuler, Department of Mathematics, University of Texas at Austin.}
\email{dylan.altschuler@austin.utexas.edu}

\date{}
\begin{document}
\begin{abstract}
Long-range expansion is a combinatorial graph property introduced to construct metric spaces with strong quantitative obstructions to low-distortion embeddings \cite{ADTT,ADTTUniversal}. Such extremal constructions are motivated by fundamental questions in metric geometry and nonlinear functional analysis. However, the only known constructions of long-range expander sequences are random regular graphs; the problem of constructing explicit sequences, especially with logarithmic girth, remains open. We resolve this question in strong form by proving that the Ramanujan condition implies long-range expansion. In particular, classical Ramanujan constructions (such as LPS graphs) give explicit long-range expander sequences with logarithmic girth. We also clarify that long-range expansion implies spectral expansion, and both implications are strict. 

While long-range expansion, the Ramanujan property, and spectral expansion are strictly separated, it is \textit{a priori} not clear how robust these gaps are. We next compare them from a dynamics perspective: how do the different notions of expansion affect the mixing rate of random walks? Our second main result is that, from a dynamics perspective, LRE and the Ramanujan property are equivalent. We show that the random walk on any LRE sequence has the same cutoff location and Gaussian limit profile as on a Ramanujan graph. This incidentally offers intermediate progress between cutoff for Ramanujan graphs---proven by Lubetzky and Peres---and the long-standing conjecture of cutoff for transitive spectral expanders. 
\end{abstract}
\maketitle
\section{Introduction}

Throughout this article, $G=(V,E)$ is a finite, simple, connected $d$-regular graph with fixed degree $d \ge 3$; it will be convenient to adopt the shorthand $q:=d-1$. We will be interested in graph sequences indexed by $n=|V|$. The adjacency matrix of $G$ will be denoted by $A$, and $P$ denotes the rescaling $P := A/d$. Then,
\[
 1=\lambda_1(P)>\lambda_2(P)\ge\cdots\ge\lambda_n(P)\ge-1.
\]
For $S\subseteq V$, define
$B(S,r)=\{y:\dist(y,S)\le r\}$, and denote the ball by $B(x,r)=B(\{x\},r)$. Logarithms are natural unless specified. We now introduce the three notions of expansion that are compared in this article.

\begin{definition}[Spectral expansion]\label{def:spectral}
For $\gamma>0$, the graph $G$ is a $\gamma$-\emph{expander}
if
\[
 1-\lambda_2(P)\ge\gamma,
\]
and an \emph{absolute} $\gamma$-\emph{expander} if
\[
 \max_{2\le i\le n}|\lambda_i(P)|\le1-\gamma.
\]
A \textit{spectral expander} sequence is a sequence of graphs $(G_n)_n$, where each $G_n$ has $n$ vertices and is a $\gamma$-expander for some constant $\gamma > 0$ that is independent of $n$. An absolute expander sequence is defined analogously.
\end{definition}

Ramanujan graphs are spectral expanders whose nontrivial adjacency eigenvalues lie in the spectral interval of the infinite $d$-regular tree.

\begin{definition}[Ramanujan expansion]
A graph is \emph{Ramanujan} if
\[
 |\lambda_i(P)|\le\frac{2\sqrt{d-1}}d
 \qquad\text{whenever }\lambda_i(P)\notin\{-1,1\}.
\]
\end{definition}

The exceptional eigenvalue $-1$ allows for bipartite Ramanujan graphs.
Finally, we introduce the notion of combinatorial expansion that is the main focus of the article:

\begin{definition}[Long-range expansion]\label{def:lre}
For $c\in(0,1]$, the graph $G$ is a $c$-\emph{long-range expander} if, for every $S\subseteq V$ and every integer $r\ge0$,
\begin{equation}\label{eq:lre}
 |B(S,r)|\ge\min\left\{\frac{3n}{4},c|S|q^r\right\}\,.
\end{equation}
A sequence is a \emph{long-range expander sequence} if the same $c>0$
works for every graph in the sequence.
\end{definition}

Here, the choice of $3/4$ is inessential: any fixed constant in $(1/2,1)$ gives an equivalent condition after changing $c$.

\begin{remark}
    Definition \ref{def:lre} slightly diverges from the formulation of long-range expansion (LRE) in \cite{ADTT}. The current presentation is substantially simpler, and the equivalence of the two definitions is proven in the appendix.
\end{remark}

The long-range expansion condition asks for growth at the branching rate $q=d-1$ of the regular tree, up to a constant factor, until most of the vertex set has been reached. The growth rate $q^r$ is maximal by degree considerations. There are a number of related ``locally tree-like'' properties considered in the literature. For example, Lubetzky--Sly \cite[Lemma~2.1]{LS} show that every ball of radius $\lfloor\frac15\log_q n\rfloor$ in a random regular graph contains at most one cycle with high probability, while Bordenave's tangle-free condition in his proof of Friedman's theorem \cite{Friedman,Bordenave} gives the same conclusion at radius $\lfloor\kappa\log_q n\rfloor$ for any fixed $\kappa<1/4$. These conditions are substantially different from LRE as they only constrain balls of radius a small multiple of $\log_q n$, and hence sets of size $n^{c}$ for $c < 1$. In contrast, LRE aims for less rigid control---for example, it allows for many short cycles---but it applies at \textit{all} radii up to the diameter of the graph and for any size set.

Finally, while the relation between Ramanujan and long-range expansion was previously unknown, it is obvious that \eqref{eq:lre} implies spectral expansion. Indeed, pick the least integer $r$ with $cq^r\ge2$ and take any set $|S|\le n/2$. Then \eqref{eq:lre} gives $|B(S,r)|\ge3|S|/2$. Since $|B(S,r)\setminus S|\le |E(S,S^c)|\sum_{j=0}^{r-1}q^j$, Cheeger's inequality \cite{LPW} supplies a uniform spectral gap.

\subsection{Explicit constructions}
The search for explicit constructions of long-range expanders is motivated by long-standing questions at the intersection of metric geometry, geometric group theory, and functional analysis about impossibility of low-distortion and coarse embeddings. The connection, quoted here in an abridged formulation, is:

\begin{theorem}[Nonlinear spectral gaps from LRE \cite{ADTT}]\label{thm:nonlinear-background}
Suppose that $(G_n)$ is a $d$-regular $c$-long-range expander sequence with $d \ge 3$ fixed.
Fix $K,M\ge1$. Let $X$ be a finite-dimensional normed space with a
$K$-unconditional basis and cotype $s\ge2$ with constant at most $M$.
There are constants $C,C'>0$ depending only on $(d,c,K,M)$ such that every $f:V(G_n)\to X$ satisfies
\begin{equation}\label{eq:nonlinear-poincare}
 \frac1{n^2}\sum_{u,v\in V(G_n)}\|f(u)-f(v)\|_X
 \le\frac{C\, s^{10}}{|E(G_n)|}
       \sum_{\{u,v\}\in E(G_n)}\|f(u)-f(v)\|_X.
\end{equation}
In particular, the distortion of any embedding of $G_n$ into $X$ is
at least $C'\, s^{-10}\log n$.
\end{theorem}

Inequalities of the form \eqref{eq:nonlinear-poincare} are referred to as ``nonlinear Poincar\'e inequalities'' and are one of the few tools for proving non-embedding theorems. Long-range expander graphs offer examples of metric spaces with strong, quantitative non-embedding properties. Earlier methods for general
spectral expanders \cite{Oz04,OS94,Na14} in general only give distortion bounds that decay super-exponentially in $s$. It remains a central question in functional analysis, posed by Pisier and Mendel--Naor \cite{Pi10,MN}, to remove the unconditionality
assumption that existing results rely on:
\begin{center}
\em Does every spectral expander sequence satisfy a nonlinear
Poincar\'e inequality, uniformly in the graph size, for each Banach
space of finite cotype? Does even one such sequence exist?
\end{center}

A closely related direction studies \textit{superexpanders}: sequences of graphs that satisfy nonlinear Poincar\'e inequalities for every uniformly convex\footnote{Uniform convexity implies finite cotype; the converse fails.} Banach space. Their existence, asked for by Kasparov and Yu \cite{KY06}, was first established by Lafforgue \cite{La08}, and then alternate constructions were given by Mendel and Naor \cite{MN}. Motivated by connections to the Novikov conjecture \cite{KY06,MN}, it remains an important open question to determine the existence of superexpanders with logarithmic girth \cite[Section~1.1]{MN}.

Together, these research directions suggest the construction of long-range expanders with logarithmic
girth as an important goal, as this would supply explicit high-girth graphs to which the quantitative theorem above applies. Previously, the known constructions of LRE sequences used random regular graphs, whose girth remains bounded in probability. 

\begin{proposition}[\cite{ADTT}; see also Appendix \ref{sec:property-B} below]
For each fixed $d\ge3$, a uniformly random $d$-regular graph
on $n$ vertices is a $c_d$-long-range expander with probability
$1-o(1)$, for a constant $c_d>0$.
\end{proposition}

The natural question of obtaining explicit non-random constructions for LRE graphs was formally posed in \cite{ADTT} and \cite{ADTTUniversal}, as well as asked in private communication by A. Naor, A. Eskenazis, and others. We resolve this in strong form by constructing explicit LRE sequences with high girth, and furthermore by showing that the Ramanujan property is strictly stronger than LRE. 

\begin{theorem}[Main result I: Ramanujan implies long-range expansion]
\label{thm:construction}
Every Ramanujan graph of degree $d\ge3$ is a $1/(16q)$-long-range expander: for every $S \subseteq V$ and every integer $r\ge0$,
\begin{equation}\label{eq:growth}
 |B(S,r)|
 \ge
 \min\left\{\frac{3n}{4},\frac{|S|q^r}{16q}\right\}.
\end{equation}
Conversely, there is an explicit, nonbipartite, vertex-transitive,
long-range expander sequence of logarithmic girth that is not Ramanujan. 
\end{theorem}

\begin{corollary}
    Any Ramanujan sequence of logarithmic girth is also a long-range expander sequence of logarithmic girth; for example, the Lubotzky--Phillips--Sarnak \cite{LPS} construction.
Morgenstern's construction \cite{Mor94} gives such sequences whenever
$d-1$ is a prime power, including $d=3$.
\end{corollary}

\begin{remark}[Other notions of combinatorial expanders]
Explicit lossless vertex expanders of ``constant'' degree\footnote{
For every $\epsilon>0$ and every sufficiently large fixed $d\ge d_0(\epsilon)$, they construct an infinite family of $d$-regular graphs with vertex expansion at least $(1-\epsilon)d$. Constructing optimal vertex expanders at prescribed degrees remains open.} have recently been
constructed \cite{HLMRZ}, following earlier progress in \cite{HLMOZ}. This is a substantially different requirement from
LRE: strong lossless vertex expansion allows a fixed multiplicative loss from the tree growth
rate to accumulate at each radius, whereas \eqref{eq:lre} permits only
a constant loss over all radii. Conversely, the constant $c$ in \eqref{eq:lre} permits
substantial losses at small radii. Indeed, even Ramanujan graphs can fail to have
near-optimal one-step vertex expansion \cite{KK22}; related examples with
logarithmic girth and an asymptotically Ramanujan spectrum were
constructed in \cite{MM20}.
\end{remark}

\subsection{Dynamics}
We turn towards comparing the three notions of expansion via their effect on the rate of mixing for random walks. The motivation is two-fold. Our first motivation is that spectral properties of graphs are notoriously delicate, prompting the question: how robust is the separation between the notions of expansion? Studying dynamics on the various classes of graphs gives us another angle to compare the notions of expansion. Our second motivation comes from the celebrated conjecture that \textit{all} vertex-transitive expanders admit cutoff (see \cite[Section~1.1]{LP}). 
\begin{conjecture}[Peres, 2004 \cite{AIMMixingTimes}]\label{conj:transitive}
    Lazy random walk on every fixed-degree vertex-transitive
    expander sequence exhibits cutoff. 
\end{conjecture}

Informally, cutoff is a sharp threshold phenomenon in which a random walk is far from stationarity until some ``cutoff'' time, at which point the walk rapidly transitions to almost perfectly mixed. Here, mixing refers to approaching the equilibrium measure in total variation distance; recall that the uniform distribution is the equilibrium measure on any connected $d$-regular graph. For a walk which has cutoff, the mixing behavior is determined to first order by the cutoff location. The more delicate second-order behavior, describing the mixing rate within the cutoff window, is called the limit profile. We now state these definitions formally.

\begin{definition}[Cutoff and limit profile]\label{def:cutoff}
Let $G_n=(V_n,E_n)$ be a sequence of connected $d$-regular graphs,
where $|V_n|=n$, with adjacency matrices $A_n$.
For simple random walk and lazy random walk, respectively, set
\[
 P_n=\frac{A_n}{d}
 \qquad\text{or}\qquad
 P_n=\frac12\left(I+\frac{A_n}{d}\right).
\]
In the non-lazy case, assume that each $G_n$ is non-bipartite.
Denote the unique stationary distribution by $\pi_n(y)=1/n$. Write
\[
 \|\mu-\nu\|_{\mathrm{TV}}
 =\frac12\sum_{y\in V_n}|\mu(y)-\nu(y)|,
 \qquad
 d_n(t)=\max_{x\in V_n}
 \|P_n^t(x,\cdot)-\pi_n\|_{\mathrm{TV}},
\]
and define
\(
 t_{\mathrm{mix}}^{(n)}(\varepsilon)
 =\min\{t\in\mathbb Z_{\ge0}:d_n(t)\le\varepsilon\}.
\)
The sequence exhibits \emph{cutoff} at times $t_n$ if
\[
 \frac{t_{\mathrm{mix}}^{(n)}(\varepsilon)}{t_n}
 \longrightarrow1
 \qquad\text{for every fixed }\varepsilon\in(0,1).
\]
A sequence $w_n\ge1$ with $w_n=o(t_n)$ is a \emph{cutoff window}
around $t_n$ if
\[
 t_{\mathrm{mix}}^{(n)}(\varepsilon)
 =t_n+O_\varepsilon(w_n)
 \qquad\text{for every fixed }\varepsilon\in(0,1).
\]
When the following limit exists for every $s\in\mathbb R$, the function $F$ is called the \emph{limit profile}:
\[
 F(s)=\lim_{n\to\infty}
 d_n\bigl(\lfloor t_n+s w_n\rfloor\bigr),
 \qquad s\in\mathbb R.
\]
\end{definition}

A breakthrough theorem of Lubetzky and Peres \cite{LP} proved cutoff and also gave an explicit limit profile for the simple random walk on any fixed-degree sequence of nonbipartite Ramanujan graphs. In particular, the cutoff occurs at
\[
 t_n=\frac{d}{d-2}\log_q n,
\]
with a Gaussian profile in a window of order $\sqrt{\log n}$.
This is the fastest possible mixing scale allowed by the degree. Indeed, a simple random walk on $G$ is the projection of a simple random walk on the infinite $d$-regular tree. Away from the root, the distance of the latter walk increases with probability $(d-1)/d$ and decreases with probability $1/d$.
Thus its distance $R_t$ from the root satisfies
\[
 R_t=\frac{d-2}{d}\,t+O_{\mathbb P}(\sqrt t),
 \qquad
 |B(x,r)|\le C_dq^r.
\]
Consequently, at time $t = c t_n$, for any constant $c < 1$, the walk is concentrated on $o(n)$ vertices and its total-variation distance from the equilibrium distribution (i.e., the uniform distribution) tends to one. The fluctuations in $R_t$ are exactly the source of the Gaussian limit profile.

The matching upper bound requires substantially more.
Even under the Ramanujan assumption, the usual spectral estimate
\[
 d_n(t)\le\frac{\sqrt n}{2}
 \left(\frac{2\sqrt q}{d}\right)^t
\]
does not recover the leading constant in $t_n$. Lubetzky and Peres instead analyze the nonbacktracking walk (NBRW). It mixes approximately at time $\log_q n$, and the Gaussian cutoff profile comes from radial fluctuations of the tree walk, i.e., from passing between the NBRW and the SRW.

Our second main result gives the same conclusion under the
strictly weaker long-range expansion hypothesis. The question of cutoff for SRW on long-range expanders was conjectured to us in private communication by K. Chawla. 

\begin{theorem}[Main result II: cutoff and Gaussian profile]\label{thm:main}
Let $(G_n)$ be a sequence of $d$-regular $c$-long-range
expanders, with fixed $d,c$, and suppose that
$\lambda_n(A_n/d)\ge-1+\eta$ for some fixed $\eta>0$.
Write $P_n=A_n/d$ and $\pi_n(y)=1/n$.
Set
\[
 t_n=\frac{d}{d-2}\log_q n,
 \qquad
 \tau_d^2=\frac{4d(d-1)}{(d-2)^3}.
\]
Then, for every fixed $z\in\mathbb R$,
\begin{equation}\label{eq:profile}
 \sup_{x\in V(G_n)}
 \left|
 \left\|P_n^{\lfloor t_n+z\sqrt{\log_q n}\rfloor}
                  (x,\cdot)-\pi_n\right\|_{\mathrm{TV}}
 -\Phi\left(-\frac{z}{\tau_d}\right)
 \right|
 \longrightarrow0,
\end{equation}
where $\Phi$ is the standard normal distribution function.
In particular, for every fixed $\varepsilon\in(0,1)$,
\[
 t_{\mix}^{(n)}(\varepsilon)
 =t_n-\tau_d\Phi^{-1}(\varepsilon)\sqrt{\log_q n}
   +o(\sqrt{\log n}).
\]
\end{theorem}

Here, the assumption on $\lambda_n$ quantitatively rules out bipartite structure, ensuring mixing. For the lazy walk, no additional spectral hypothesis is needed:
\begin{corollary}[Cutoff for lazy walk]\label{cor:lazy}
The lazy random walk on every fixed-degree long-range expander sequence has 
cutoff at $\frac{2d}{d-2}\log_q n$. Its limit profile is given by \eqref{eq:profile}, with $P_n$ replaced by $(I+A_n/d)/2$, centered at $2t_n$ instead of $t_n$, and with $\tau_d$ replaced by
\[
    \tau_{d,\mathrm L}=\sqrt{\frac{2d(d^2+4d-4)}{(d-2)^3}}.
\]
\end{corollary}

The core challenge is to use crude volume growth to control how unevenly nonbacktracking paths of a fixed length are distributed among their endpoints.

\subsection*{Acknowledgments}
This work was initiated at the 2025 AIM Workshop ``Metric Embeddings'' at Caltech. There, the author received numerous helpful comments on LRE that prompted parts of the current work, especially from A. Naor, A. Eskenazis, T. Hutchcroft, and K. Chawla. In particular, K. Chawla specifically asked us the question of whether LRE graphs have cutoff, and also suggested using the entropic method of Ozawa. The author also greatly benefited from discussing long-range expansion with P. Dodos, K. Tyros, and K. Tikhomirov on many occasions. 

\subsection*{AI disclosure}
The research direction, result statements, and exposition are human. Both human and AI contributed substantially to the development of each of the proofs. In particular, conversations (spanning a couple hundred messages) were had with Opus 4 and GPT 5.4-6, which were specifically valuable for understanding and adapting the methods in the literature. The author assumes accountability for the final version of the text.

\section{Explicit constructions}
This section proves Theorem \ref{thm:construction}.

\subsection{From Ramanujan to long-range expansion}
Assume first that $G$ is nonbipartite. For any polynomial $F$ of degree at most
$r$, note that the vector $f=F(A)\mathbf1_S$ is supported on $B(S,r)$:
each power $A^k$ only connects vertices at distance at most $k$.
Consequently, Cauchy--Schwarz gives
\[
 |B(S,r)|\ge\frac{\|f\|_1^2}{\|f\|_2^2},
\]
where the norms use counting measure. Regularity and the
Ramanujan bound imply
\[
 \|f\|_1\ge\left|\sum_x f(x)\right|=|F(d)|\,|S|,
 \qquad
 \|f\|_2^2\le
 \frac{F(d)^2|S|^2}{n}
 +|S|\max_{|\lambda|\le2\sqrt q}|F(\lambda)|^2.
\]
Here the second inequality follows by decomposing
$\mathbf1_S$ into its constant and mean-zero parts.

As an example, taking the naive choice $F(x)=x^r$ gives
\[
 |B(S,r)|
 \ge \frac{|S|}{|S|/n+(2\sqrt q/d)^{2r}}
 \ge \frac12\min\left\{n,\,
 |S|\left(\frac{d^2}{4q}\right)^r\right\}.
\]
Since $d^2/(4q)<q$, this growth rate is too small to establish
long-range expansion. The idea is to improve the estimate by replacing $A^r$ with a better
polynomial in $A$. In particular, we seek a polynomial whose value at $d$ is as large as possible
compared with its values on $[-2\sqrt q,2\sqrt q]$. Chebyshev polynomials solve precisely this extremal problem.

\begin{proof}[Proof of \eqref{eq:growth}]
Let $T_r$ denote the Chebyshev polynomial of the first kind,
characterized by $T_r(\cos\theta)=\cos(r\theta)$, and set
\[
 F_r(x)=2q^{r/2}T_r\left(\frac{x}{2\sqrt q}\right).
\]
The change of variables sends the Ramanujan interval to
$[-1,1]$, where $|T_r|\le1$. The prefactor makes $F_r$ monic
for $r\ge1$ and gives the recurrence
\[
 F_0=2,\qquad F_1=x,\qquad
 F_{r+1}=xF_r-qF_{r-1}.
\]
To evaluate at the trivial eigenvalue, observe that
\[
 \frac{d}{2\sqrt q}
 =\cosh\left(\frac{\log q}{2}\right).
\]
Using $T_r(\cosh u)=\cosh(ru)$, we obtain
\[
 F_r(d)=q^r+1,
 \qquad
 \max_{|\lambda|\le2\sqrt q}|F_r(\lambda)|
 \le2q^{r/2}.
\]
The constant component therefore grows like $q^r$, while
the mean-zero component grows at most like $q^{r/2}$.
Squaring their ratio in the $\ell_1$--$\ell_2$ estimate
produces exactly the desired volume growth:
\[
 |B(S,r)|
 \ge
 \frac{(q^r+1)^2|S|^2}
 {(q^r+1)^2|S|^2/n+4q^r|S|}
 \ge
 \frac{n|S|q^r}{|S|q^r+4n} \ge \min\left\{\frac{3n}{4},\frac{q^r|S|}{16}\right\}.
\]
For bipartite graphs, use instead the degree-$r$ polynomial
$(x+d)F_{r-1}(x)/(2d)$. Its additional factor equals one
at $d$ and zero at $-d$, removing the exceptional negative
eigenvalue at the cost of changing the polynomial degree by one. Its value at $d$ is $q^{r-1}+1$, and its absolute value on
$[-2\sqrt q,2\sqrt q]$ is at most $2q^{(r-1)/2}$.
The same calculation, with $r-1$ in place of $r$, gives
$|B(S,r)|\ge\min\{3n/4,|S|q^{r-1}/16\}$, proving
\eqref{eq:growth}.
\end{proof}

\begin{remark}[Related methods]
Computing distances in graphs via polynomials of the adjacency matrix is a classical technique; see, e.g., 
\cite[Theorem~2.1]{vD98} and also \cite{LPS,CFM}. This method is also closely related to spectral
analysis of the NBRW; in particular, \cite{ABLS} use
path-counting polynomials that satisfy the same recurrence as ours,
but involve Chebyshev polynomials of the second kind.
\end{remark}

\subsection{Failure of the converse}
We construct a sequence of nonbipartite, vertex-transitive
long-range expanders of logarithmic girth that are not Ramanujan.
The Lubotzky--Phillips--Sarnak construction with parameter $p=13$
provides an infinite sequence of nonbipartite $14$-regular Ramanujan
Cayley graphs of logarithmic girth \cite{LPS}.
Fix one such graph $H$, and take two copies $H_1,H_2$.

Choose a generator $g$ and its inverse (such that $g$ and $g^{-1}$ are distinct). All edges corresponding to the other generators are left unchanged within $H_1$ and $H_2$. However, edges corresponding to these two generators
switch copies of the graph. That is, if $v_1 \in V(H_1)$ and $v_2 \in V(H_2)$ correspond to copies of the same vertex, then we replace the edges $(v_1,gv_1 ) \in E(H_1)$ and $(v_2,gv_2) \in E(H_2)$ with $(v_1,gv_2)$ and $(v_2,gv_1)$, and similarly for $g^{-1}$. 

The resulting graph $G$ is still $14$-regular. Further, since $H$ is a
Cayley graph on some group $K$, it follows that $G$ is a Cayley graph on
$K\times\mathbb Z/2 \mathbb Z$, and hence is still vertex-transitive.
However, $G$ is not Ramanujan: the function which is $+1$ on $V(H_1)$ and $-1$ on $V(H_2)$ is an eigenvector of the adjacency matrix, with eigenvalue
\[
 12-2=10>2\sqrt{13}.
\]

Next, we verify absolute expansion. The idea is to view the spectrum of $G$ as a small perturbation of the spectrum of two disjoint copies of $H$. Let $A_0$ be the adjacency matrix of two disjoint copies of $H$. Each vertex loses two neighbors
and gains two new neighbors when we form $G$. Thus the symmetric
matrix $A_G-A_0$ has absolute row sums $4$, giving
$\|A_G-A_0\|_{2\to2}\le4$.

Functions which are constant on each copy (i.e., of the form $f(v_i) = c_i$ for all $v_i \in H_i$, for $i \in \{1,2\}$, with $c_1$ and $c_2$ possibly equal) form an invariant
two-dimensional space with eigenvalues $14$ and $10$.
The orthogonal complement of this space consists of
functions whose sum on each copy is zero. For every such $f$,
the Ramanujan bound for $H$ gives
\[
 \|A_Gf\|_2
 \le \|A_0f\|_2+\|(A_G-A_0)f\|_2
 \le (2\sqrt{13}+4)\|f\|_2.
\]
Hence $14$ is a simple eigenvalue and every other eigenvalue
has modulus at most $2\sqrt{13}+4<14$. Therefore $G$ is connected,
nonbipartite, and has a uniform absolute spectral gap.

Next, it is immediate from construction that $G$ has at least the same girth as $H$, which has logarithmic girth. Finally, as $H$ is Ramanujan, it satisfies LRE with the parameters in \eqref{eq:growth}. Letting  $p:V(G)\to V(H)$ denote the covering projection from $G$ to $H$, we have for any $S \subset V(G)$
\[
    p(B_G(S,r))=B_H(p(S),r).
\]
Since $|p(S)|\ge|S|/2$, the long-range growth bound for
$H$ implies that, for some absolute constant $c > 0$,
\[
 |B_G(S,r)|
 \ge c\min\{|V(H)|,|p(S)|13^r\}
 \ge \frac{c}{2} \min\{|V(G)|,|S|13^r\}.
\]
This weak growth bound differs from \eqref{eq:lre} in the location of the implicit constant. However, we just proved above that $G$ has an absolute uniform spectral gap. Cheeger's inequality then implies for some absolute constant $r_0$, every set $S$ with $|S| \ge (c/2)|V(G)|$ satisfies
\[
    B_G(S,r_0) \ge \frac{3}{4}|V(G)|.
\]
Hence \eqref{eq:lre} follows with expansion constant $c/(2 \cdot 13^{r_0})$ so that $G$ is indeed a long-range expander. This completes the proof of Theorem~\ref{thm:construction}. \qed

\subsection{Spectral expansion does not imply long-range expansion}
Let $(H_m)$ be a $d$-regular expander sequence and let
$G_m$ denote the $(d+1)$-regular graph consisting of two
copies of $H_m$, with an edge joining each vertex to its counterpart. The adjacency matrix of $G_m$ has eigenvalues $\lambda_i(A_{H_m})\pm1$, so it has a
uniform one-sided spectral gap. However,
\[
 |B_{G_m}((x,i),r)|\le2|B_{H_m}(x,r)|\le C_d(d-1)^r,
\]
so that $(G_m)_m$ is not a long-range expander sequence of degree $d+1$. Taking $r=\lfloor\tfrac12\log_d|V(G_m)|\rfloor$
contradicts \eqref{eq:lre} for every fixed $c>0$.

\section{Cutoff}\label{sec:cutoff-proof}

We characterize the mixing time of the random walk on an LRE sequence by analyzing the nonbacktracking random walk (NBRW), which never immediately reverses an edge. This is the strategy utilized for Ramanujan graphs in Lubetzky--Peres \cite{LP}, but we replace their spectral analysis of its transition operator by an entropy estimate, inspired by Ozawa's entropic proof \cite{Ozawa} of the Lubetzky--Peres result. Whereas Ozawa uses the Ramanujan spectral bound to control the entropy of the simple random walk, we instead use combinatorial arguments enabled by LRE. The main challenge and contribution is to show that long-range expansion
forces the entropy of the endpoint of a NBRW to be near-maximal and hence close to uniform in KL divergence.

Of course, a discrete random variable with near-maximal entropy can still be very far from the uniform distribution in total variation distance. Fortunately, a standard argument using weak spectral expansion shows that a few extra steps of the random walk ``smooth out'' the distribution of the NBRW enough to obtain total variation bounds from the entropic estimates, yielding cutoff.

\subsection{The NBRW}\label{sec:nonbacktracking}
We begin with some notation. Let $M_j(x,y)$ be the number of nonbacktracking paths of length $j$ from $x$ to $y$, and set
\[
     D_j=dq^{j-1},\qquad K_j(x,y)=\frac{M_j(x,y)}{D_j}
     \quad(j\ge1),\qquad K_0(x,\cdot)=\delta_x.
\]
That is, $K_j(x,\cdot)$ is the endpoint law of the NBRW. Also denote the first radius at which there are at least $n$ paths by
\[
 j_n=\min\{j\ge1:dq^{j-1}\ge n\}=\log_q n+O_d(1).
\]
Finally, for a probability law $\mu$ on $V$, write $H(\mu)=-\sum_y\mu(y)\log\mu(y)$ with the convention $0\log0:=0$. The main estimate is the following.

\begin{proposition}[Endpoint entropy]\label{prop:entropy}
Let $G$ be a $c$-long-range expander. There is a constant $C=C(d,c)$ such that, for all sufficiently large $n$,
\begin{equation}\label{eq:criticalentropy}
 H(K_j(x,\cdot))\ge\log n-C\log\log n
 \qquad(x\in V,\ j\ge j_n).
\end{equation}
In particular,
\begin{equation}\label{eq:fibers}
 \mathbb E\log M_{j_n}(x,Y)\le C\log\log n,
 \qquad Y\sim K_{j_n}(x,\cdot).
\end{equation}
\end{proposition}

For the latter condition, the expectation in \eqref{eq:fibers} weights an endpoint in proportion to the number of paths reaching it. That is, \eqref{eq:fibers} roughly says that a sampled NBRW path has few alternative nonbacktracking paths of the same length between its endpoints, where fluctuations are measured on the exponential scale. The main challenge is transferring volume growth, which only controls how many endpoints are available, to an entropic control of how unevenly the paths are distributed among them.
For example, $\mu=\frac12\delta_x+\frac12\pi$ has full support but far from maximal entropy:
$H(\mu)=\frac12\log n+O(1)$.

In order to utilize Markov chain mixing time techniques, we need to make the non-backtracking walk into a\textit{ bona fide }Markov chain. We use the standard approach of tracking oriented
edges rather than vertices. The state space and transition matrix are
\[
 \Omega=\{(u,v):u\sim v\},\qquad N=dn,\qquad
 Q((u,v),(v,w))=\frac1q\quad(w\ne u).
\]
Every state has $q$ predecessors and $q$ successors, so the uniform
law on $\Omega$ is stationary. For $F\subseteq\Omega$, let
$B(F,r) := B_\Omega(F,r)$ be its radius-$r$ ball, with distances measured along
directed edges. We first transfer the LRE volume growth condition to this
directed graph.

\begin{lemma}\label{lem:transfer}
There is $c_*=c_*(d,c)>0$ such that every $c$-long-range expander satisfies, for every $F \subseteq \Omega$ and every integer $r\ge0$,
\begin{equation}\label{eq:edgegrowth}
    |B(F,r)|\ge c_*\min\{N,|F|q^r\}.
\end{equation}
\end{lemma}

\begin{proof}
We may assume $c\le3/4$. Choose a fixed integer $k\ge1$ with $q^{-k}\le c(q-1)/2$.
Let $C$ be the terminal vertices of $B(F,2k)$,
and let $U\subseteq C$ consist of those not reached at any
positive time up to $2k$.

For each $u\in U$, choose an edge of $F$ ending at $u$.
The $|U|q^k$ nonbacktracking continuations of these edges
of length $k$ have distinct endpoints. Indeed, if two
distinct paths met, delete their common terminal segment
and follow the first remaining path and then the second
backwards. Since their lengths agree, this gives a nonempty
nonbacktracking walk reaching a vertex of $U$ in at most
$2k$ steps, a contradiction. Thus
\[
 |U|\le q^{-k}|C|.
\]

Now take a shortest path from $C$ to a vertex of $B(C,\ell)$
not reached by time $2k+\ell$. Its starting vertex must lie
in $U$: otherwise a positive-time arrival has its preceding
vertex in $C$, whereas the shortest path immediately leaves
$C$, so the two paths can be concatenated without backtracking.
Moreover, its first step must reverse the chosen initial
edge. Each vertex of $U$ therefore accounts for at most
$1+q+\cdots+q^{\ell-1}$ exceptional endpoints. Consequently,
whenever $|C|q^\ell\le n$,
\[
 |B(F,2k+\ell)|
 \ge |B(C,\ell)|-|U|\frac{q^\ell-1}{q-1}
 \ge \frac c2|C|q^\ell.
\]
At $\ell=\lfloor\log_q(n/|C|)\rfloor$, this lower bound
is at least $cn/(2q)$; larger radii follow by monotonicity.
Since $|C|\ge |F|/d$ and $k$ depends only on $d,c$,
these bounds imply \eqref{eq:edgegrowth}.
For radii below $2k$, use $|B(F,r)|\ge |F|$.
\end{proof}
\subsection{From set growth to entropy}\label{sec:entropy}

We now turn to the entropy estimate for the simple random walk on $\Omega$. The main idea is to exploit the following inductive structure. Choose a forest of shortest paths from the set of NBRW paths so that each reachable state is connected to the starting set by a unique path on this forest. (There may be multiple valid choices of the forest; one is chosen arbitrarily). A NBRW conditioned to walk on this forest until time $t$ will be uniformly distributed on some shell of radius $t$. The uniformity is crucial and allows us to obtain a near-maximal lower-bound on the entropy of the walk at time $t$, using long-range expansion. 

So, we can run the NBRW until it first leaves the forest and lower-bound the entropy of the walk at its exit time (after arguing uniformity of its exit location under suitable conditioning). Then, we restart the walk, construct a new forest, and repeat, inductively adding up the entropy of the NBRW on these increments. Long-range expansion is also used to upper bound the probability of exiting the forest, yielding control over the number and length of these increments. We now make this argument precise.

\begin{lemma}\label{lem:directedentropy}
Let $\Omega$ be the vertex set of a directed graph on $N$ states,
with $q\ge2$ incoming and $q$ outgoing edges at every state. Suppose that for all $\emptyset \neq S \subseteq \Omega$ and every integer $r\ge0$,
\[
 |B(S,r)|\ge a\min\{N,|S|q^r\}.
\]
Then, there are constants $C_0\ge1$ and $\delta>0$, depending only on
$a,q$, such that a simple random walk $(Y_t)_{t\ge0}$ started uniformly on
$\emptyset \neq S \subseteq \Omega$ satisfies, for every integer $h\ge0$,
\begin{equation}\label{eq:entropybound}
 H(Y_h)\ge\log|S|+h\log q-\delta^{-1}\log(qh+1)
 \qquad\text{if }h+C_0\le\log_q(N/|S|).
\end{equation}
\end{lemma}

\begin{proof}
Assume $a\le1$, choose an integer $C_0\ge1$ with
$q^{-C_0}\le a(q-1)/2$, and put $\delta=a(q-1)/(2q)$.
We induct on $h$, simultaneously for all eligible starting sets $S$.
The assertion is immediate when $h=0$.

Write $\Gamma_r(S)$ for the states at directed distance exactly $r$
from $S$. Each state of $\Gamma_{r+1}(S)$ has a predecessor in
$\Gamma_r(S)$, so $|\Gamma_{r+1}(S)|\le q|\Gamma_r(S)|$.
Using growth at radius $h+C_0$, we obtain
\[
 \begin{aligned}
 a|S|q^{h+C_0} \le |B(S,h+C_0)| 
 &\le |S|\sum_{r=0}^{h-1}q^r
       +|\Gamma_h(S)|\sum_{r=0}^{C_0}q^r \le\frac{|S|q^h+|\Gamma_h(S)|q^{C_0+1}}{q-1}.
 \end{aligned}
\]
Thus
\begin{equation}\label{eq:layer}
 |\Gamma_h(S)|\ge\delta|S|q^h.
\end{equation}

For each state in $\Gamma_r(S)$, $r\ge1$, choose one incoming edge
from $\Gamma_{r-1}(S)$. The chosen edges form a forest rooted at $S$:
there is exactly one forest path from $S$ to each reachable state.
Let $T$ be the first time the walk uses an edge outside this forest.
All length-$h$ trajectories have probability $1/(|S|q^h)$, so
\begin{equation}\label{eq:success}
 \alpha :=\Pr(T>h)=\frac{|\Gamma_h(S)|}{|S|q^h}\ge\delta,
 \qquad \text{ and } \qquad
 \mathcal L(Y_h\mid T>h)=\operatorname{Unif}(\Gamma_h(S)).
\end{equation}

On departure from the forest, conditioning only on $\{T=t\}$ need not
leave a uniform law: several departing edges may have the same
endpoint. To distinguish them, list the nonforest edges from
$\Gamma_{t-1}(S)$ into each state $y$ in an arbitrary fixed order.
Let $m_t(y)\le q$ be their number, and let
\[
 S_{t,i}=\{y:m_t(y)\ge i\},\qquad 1\le i\le q.
\]
We consider $1\le t\le h$ and omit pairs $(t,i)$ of probability zero.
On $\{T=t\}$, let $I$ be the index of the edge used in the list at
$Y_t$. For each $y\in S_{t,i}$, exactly one trajectory realizes
$\{T=t,I=i,Y_t=y\}$: follow the forest to the selected edge and
then traverse it. Therefore
\begin{equation}\label{eq:branch}
 \beta_{t,i}:=\Pr(T=t,I=i)=\frac{|S_{t,i}|}{|S|q^t},
 \qquad \text{ and } \qquad
 \mathcal L(Y_t\mid T=t,I=i)=\operatorname{Unif}(S_{t,i}).
\end{equation}
This conditioning involves only the first
$t$ steps, so the remaining walk has its usual transition law.
Moreover, $|S_{t,i}|\le|S|q^t$ implies
\[
 h-t+C_0\le\log_q(N/|S|)-t\le\log_q(N/|S_{t,i}|).
\]
The induction hypothesis consequently applies to the remaining
$h-t$ steps, giving
\[
 H(Y_h\mid T=t,I=i)
 \ge \log|S|+h\log q+\log\beta_{t,i}
      -\delta^{-1}\log(q(h-t)+1).
\]
On the event $\{T>h\}$, \eqref{eq:success} gives the conditional entropy of $Y_h$ as
$\log|S|+h\log q+\log\alpha$.

Let $Z$ indicate whether the walk stays in the forest through
time $h$, and otherwise specify its departure time and edge index:
\[
 Z=
 \begin{cases}
  0,&T>h,\\
  (T,I),&T\le h.
 \end{cases}
\]
Thus $\Pr(Z=0)=\alpha$ and $\Pr(Z=(t,i))=\beta_{t,i}$.
Since $Z$ takes at most $qh+1$ values,
\[
 H(Z)=-\alpha\log\alpha-\sum_{t,i}\beta_{t,i}\log\beta_{t,i}
 \le\log(qh+1).
\]

Conditioning on $Z$ separates the walk into the cases
estimated above. By concavity of entropy,
\[
 H(Y_h)
 \ge H(Y_h\mid Z) =\alpha H(Y_h\mid T>h)
   +\sum_{t,i}\beta_{t,i}H(Y_h\mid T=t,I=i).
\]
Substituting the preceding conditional entropy bounds, and using
$\alpha+\sum_{t,i}\beta_{t,i}=1$, gives

\begin{align*}
 H(Y_h)
 &\ge \log|S|+h\log q
      +\alpha\log\alpha+\sum_{t,i}\beta_{t,i}\log\beta_{t,i}
      -\delta^{-1}\sum_{t,i}\beta_{t,i}\log(q(h-t)+1)\\
 &=\log|S|+h\log q-H(Z)
      -\delta^{-1}\sum_{t,i}\beta_{t,i}\log(q(h-t)+1)\,.
\end{align*}
Recall that $H(Z)\le\log(qh+1)$ and
$\sum_{t,i}\beta_{t,i}=1-\alpha\le1-\delta$. Using these relations and the crude inequality $\log(q(h-t)+1)\le\log(qh+1)$,
\[
 \begin{aligned}
 H(Y_h)  &\ge\log|S|+h\log q-\log(qh+1)
      -\frac{1-\delta}{\delta}\log(qh+1)\\
 &=\log|S|+h\log q-\delta^{-1}\log(qh+1).
 \end{aligned}
\]

This completes the induction.
\end{proof}

\begin{proof}[Proof of Proposition~\ref{prop:entropy}]
Apply Lemmas~\ref{lem:transfer} and \ref{lem:directedentropy} to the
chain on oriented edges. For a sufficiently large fixed integer
$C_0$, set $k_0=\lfloor\log_q(dn)\rfloor-C_0$, choosing $C_0$
large enough that $k_0\le j_n-1$.
Uniformly in the initial oriented edge $e$,
\[
 H(Q^{k_0}(e,\cdot))\ge\log(dn)-C\log\log n.
\]
Because $Q$ is doubly stochastic, concavity of $-u\log u$ gives
$H(\mu Q)\ge H(\mu)$ for every law $\mu$.
A vertex-started walk first chooses one of $d$ oriented edges
uniformly and then makes $j-1$ transitions. By entropy concavity
and monotonicity, its final oriented edge $(U,Y)$ satisfies
\[
 H(U,Y)\ge\log(dn)-C\log\log n\qquad(j\ge j_n).
\]
There are at most $d$ possible predecessors of $Y$, whence
\[
 H(Y)=H(U,Y)-H(U\mid Y)\ge\log n-C\log\log n.
\]
Finally,
\[
 H(K_{j_n}(x,\cdot))
 =\log D_{j_n}-\mathbb E\log M_{j_n}(x,Y),
 \qquad n\le D_{j_n}<qn,
\]
which also proves \eqref{eq:fibers}.
\end{proof}

This completes the conceptual contributions of this article; the remaining
steps combine Ozawa's entropy-to-mixing argument \cite{Ozawa} with the covering-tree
reduction and radial central limit theorem of
\cite[Sections~2.3 and~5.2]{LS} and \cite[Section~2.1]{LP}. 

\subsection{From entropy to mixing}\label{sec:profile}
Near-maximal entropy and a spectral gap imply rapid mixing, as the
following consequence of Ozawa's entropy-production estimate shows.

\begin{lemma}\label{lem:smoothing}
Let $P$ be a symmetric stochastic matrix on $n$ states, all of whose
nonconstant eigenvalues have absolute value at most $1-\eta$, for
some fixed $\eta>0$, and let $\pi$ be the uniform measure. For every fixed $C > 0$,
\[
 \sup_{\mu:\ H(\mu)\ge\log n-C\log\log n}
 \|\mu P^{s_n}-\pi\|_{\TV}=o(1),
 \qquad s_n=\lceil (\log\log n)^2\rceil.
\]
\end{lemma}

\begin{proof}
Fix a probability law $\mu$ satisfying the entropy bound.
For each $k\ge0$, the law after $k$ steps is $\mu P^k$.
Applying \cite[Lemma~2]{Ozawa} to this law gives
\[
 H(\mu P^{k+1})-H(\mu P^k)
 \ge\frac\eta4\|\mu P^k-\pi\|_{\TV}^2.
\]
The right-hand side is decreasing in $k$. Summing over $0\le k<s$ and using $H(\mu P^s)\le\log n$,
we therefore obtain
\[
 \frac{\eta s}{4}\|\mu P^s-\pi\|_{\TV}^2
 \le H(\mu P^s)-H(\mu)
 \le\log n-H(\mu)
 \le C\log\log n.
\]
Taking any $s=\omega(\log\log n)$, e.g., $s = \lceil (\log \log n)^2 \rceil$ for concreteness, proves the claim.
\end{proof}

\subsection{The covering tree and Gaussian profile}

Fix $x\in V$ and realize the infinite rooted $d$-regular tree as
the tree of finite nonbacktracking paths starting at $x$. Projection
maps each such path to its endpoint in $G$ and sends simple random
walk on the tree to simple random walk on $G$. For a simple random walk on the infinite tree, let $R_t$ be its distance from the root. Conditional on $R_t=j$,
the tree endpoint is uniform on the sphere of radius $j$, by
symmetry. The projected endpoint consequently has law $K_j(x,\cdot)$,
and hence
\begin{equation}\label{eq:treemixture}
 P^t(x,\cdot)=\sum_{j=0}^t\Pr(R_t=j)K_j(x,\cdot).
\end{equation}
Away from zero, $R_t$ increases with probability $q/d$ and decreases
with probability $1/d$. Its mean increment and increment variance are
\[
 v=\frac{d-2}{d},\qquad \sigma_R^2=\frac{4q}{d^2}.
\]
The probability of returning from one to zero is $1/q$, as follows
from standard hitting-probability recurrence arguments. Therefore the number
of visits to zero is almost surely finite, and hence the increments in $R_t$ can be coupled to independent copies of the variable that takes the values $\{-1,+1\}$ with probabilities $1/d,q/d$, respectively, up to an almost surely bounded error. The central limit theorem then gives
\begin{equation}\label{eq:radialclt}
 \frac{R_t-vt}{\sigma_R\sqrt t}\Longrightarrow\mathcal N(0,1).
\end{equation}

\begin{proof}[Proof of Theorem~\ref{thm:main}]
Write $P=P_n$, $\pi=\pi_n$, and $L_n=\log_q n$.
The computation in the introduction gives
$1-\lambda_2(P)\ge\gamma$ for some fixed $\gamma>0$.
Together with the assumption on the smallest eigenvalue, this
bounds all nonconstant eigenvalues in absolute value by
$1-\min\{\gamma,\eta\}$. Thus Lemma~\ref{lem:smoothing} applies.

Set $s_n=\lceil(\log\log n)^2\rceil$.
By Proposition~\ref{prop:entropy}, every distribution
$K_j(x,\cdot)$ with $j\ge j_n$ has entropy at least
$\log n-C\log\log n$. After another $s_n$ simple random walk
steps, Lemma~\ref{lem:smoothing} puts each of these distributions
within $o(1)$ of uniform, uniformly in $x$ and $j$.
These additional steps are negligible compared with the cutoff
window, since $s_n=o(\sqrt{L_n})$.

For $t\ge s_n$, apply \eqref{eq:treemixture} to the first
$t-s_n$ steps and then run the walk for another $s_n$ steps:
\[
 P^t(x,\cdot)
 =\sum_{j=0}^{t-s_n}\Pr(R_{t-s_n}=j)
                       K_j(x,\cdot)P^{s_n}.
\]
The terms with $j\ge j_n$ are within $o(1)$ of uniform;
the remaining terms have total-variation distance at most one.
Convexity of total variation therefore gives
\begin{equation}\label{eq:upper}
 \|P^t(x,\cdot)-\pi\|_{\TV}
 \le\Pr(R_{t-s_n}<j_n)+o(1).
\end{equation}

For the lower bound, set $r_n=\lfloor L_n-s_n\rfloor$.
The degree bound
\[
 |B(x,r)|\le1+d\frac{q^r-1}{q-1}
\]
implies $\pi(B(x,r_n))=O(q^{-s_n})=o(1)$.
If $R_t\le r_n$, then the projected endpoint lies in $B_G(x,r_n)$. Hence
\begin{equation}\label{eq:lower}
 \|P^t(x,\cdot)-\pi\|_{\TV}
 \ge\Pr(R_t\le r_n)-o(1).
\end{equation}
Take $t=\lfloor L_n/v+z\sqrt{L_n}\rfloor$.
Since $j_n=L_n+O(1)$ and $s_n=o(\sqrt{L_n})$,
the two probabilities in \eqref{eq:upper}--\eqref{eq:lower}
converge by \eqref{eq:radialclt} to
\[
 \Phi\left(-\frac{zv^{3/2}}{\sigma_R}\right)
 =\Phi\left(-\frac z{\tau_d}\right),\qquad
 \tau_d^2=\frac{\sigma_R^2}{v^3}
 =\frac{4d(d-1)}{(d-2)^3}.
\]
The errors are uniform over $x$, proving \eqref{eq:profile}.
The mixing-time formula, and hence cutoff, follow by monotonicity
of total variation and continuity of the strictly decreasing profile.
\end{proof}

\begin{proof}[Proof of Corollary~\ref{cor:lazy}]
The computation in the introduction gives
$1-\lambda_2(P)\ge\gamma$ for some fixed $\gamma>0$.
Thus all nonconstant eigenvalues of
$P_{\mathrm L}=(I+P)/2$ lie in $[0,1-\gamma/2]$,
so Lemma~\ref{lem:smoothing} applies to the lazy walk.

The same covering-tree representation remains valid for the lazy version of the random walk, with the same
conditional endpoint laws $K_j(x,\cdot)$. Hence the entropy estimate
and the upper and lower bounds in the preceding proof still apply.
Away from zero, the lazy tree walk has radial increments in
$\{+1,-1,0\}$ with probabilities $q/(2d),1/(2d),1/2$, respectively.
Their mean is $v/2$ and their variance is $1/2-v^2/4$.
The total time spent at zero is again almost surely finite, so
the radial central limit theorem holds with these constants.
The center is therefore $2\log_q n/v$, and
\[
 \tau_{d,\mathrm L}^2
 =\frac{1/2-v^2/4}{(v/2)^3}
 =\frac{2d(d^2+4d-4)}{(d-2)^3}.\qedhere
\]
\end{proof}

\appendix

\section{Comparison of long-range expansion conditions}\label{sec:property-B}

In \cite[Definition~1.5]{ADTT}, long-range expansion was presented as a combination of \eqref{eq:lre}, called property~(A), together with an additional incidence condition, called property~(B). Property (B), which is technical and slightly unwieldy, was chosen specifically for applications to nonlinear Poincar\'e inequalities. Properties (A) and (B) were then proven to hold with high probability for random $d$-regular graphs with $d \ge 3$ and $d \ge 6$, respectively, only yielding LRE sequences of degree at least six. 

We now show that property (B) follows from a uniform one-sided spectral gap, which both establishes the equivalence of the current and previous definitions of LRE (up to changing the implicit constants), and also 
shows that random $d$-regular graphs indeed have LRE for every $d\ge3$, resolving \cite[Problem~6.3]{ADTT}. We now define property (B).

\begin{definition}
    Fix parameters $\alpha,\varepsilon\in(0,1]$ and $L\ge1$. Say that $G = (V,E)$ satisfies property (B) if, for every nonempty $S\subseteq V$ and integer $\ell\ge1$
    satisfying $\alpha q^{\ell-1}|S|\le3n/4$,  some $v\in S$ satisfies
    \[
     \bigl|\{e\in T:\dist(v,e)\le\ell-1\}\bigr|
     \le L(q-\varepsilon)^\ell,
    \]
    where
    \[
    T=\left\{e\in E:
     \bigl|\{v\in S:\dist(v,e)\le\ell-1\}\bigr|
     \ge L(q-\varepsilon)^\ell\right\}.
    \]
\end{definition}

\begin{lemma}\label{lem:property-B}
Fix $d\ge3$, $\gamma>0$, and $\alpha\in(0,1]$.
There are $\varepsilon\in(0,1]$ and $L\ge1$, depending only
on $d,\gamma,\alpha$, such that every $d$-regular graph with
$1-\lambda_2(P)\ge\gamma$ satisfies property~(B) with $(\alpha,\varepsilon,L)$.
\end{lemma}

The argument below extends \cite[Proposition~1.9(b)]{ADTT} from its
stronger spectral hypothesis to an arbitrary one-sided spectral gap.
We retain the final averaging argument of \cite[Section~5.2]{ADTT},
but bound the relevant second moment using tree hitting probabilities. All norms in this proof are with respect to the counting measure.

\begin{proof}
Fix $\ell\ge1$ and $S\subseteq V$ with
$\alpha q^{\ell-1}|S|\le3n/4$.
For an edge $e=\{u,w\}$, let $b_e$ count the vertices of $S$
within distance $\ell-1$ of at least one of $u,w$.
We will bound $\sum_e b_e^2$ and then average over $S$.

Let $P_{\mathrm L}=(I+P)/2$ be the transition matrix of lazy
random walk on $G$. The spectral-gap assumption gives
\[
 \|P_{\mathrm L}^t g\|_2\le\rho^t\|g\|_2
 \quad\text{whenever }\sum_{x\in V}g(x)=0,
 \qquad \rho=\max\{1/2,1-\gamma/2\}<1.
\]

First, consider lazy random walk on the infinite $d$-regular
tree. A specified vertex at distance $k$ is ever reached with
probability $q^{-k}$. Conditional on reaching that vertex, the
walk moves toward it with probability $q/(2d)$, away from it
with probability $1/(2d)$, and stays put with probability $1/2$.
These conditional probabilities follow directly from $q^{-k}$.
The distance therefore decreases in expectation by $(d-2)/(2d)$
per step, giving conditional expected hitting time $2dk/(d-2)$.
Choose an integer $K\ge4d/(d-2)$.
Markov's inequality shows that the target is reached by time
$Kk$ with probability at least $q^{-k}/2$.
Projecting to $G$, and noting that reaching a vertex at distance
$k$ takes at least $k$ steps, gives
\[
 \sum_{t=k}^{Kk}P_{\mathrm L}^t(x,y)\ge\tfrac12q^{-k}
 \qquad\text{if }\dist(x,y)=k.
\]
The displayed bound also holds for $k=0$.

Write $f(v)=|S\cap B(v,\ell-1)|$. Summing this bound over
the vertices of $S$ at distance at most $\ell-1$ from $v$ gives
\[
 f\le 2\sum_{k=0}^{\ell-1}q^k
             \sum_{t=k}^{Kk}P_{\mathrm L}^t\mathbf1_S.
\]
Next, as in \cite{ADTT}, decompose $\mathbf1_S=(|S|/n)\mathbf1+g$, where
$\sum_xg(x)=0$ and $\|g\|_2\le\sqrt{|S|}$.
The constant part and the spectral estimate above give
\[
 \begin{aligned}
 \|f\|_2
 &\le C\ell q^\ell\frac{|S|}{\sqrt n}
       +\frac2{1-\rho}\sum_{k=0}^{\ell-1}(q\rho)^k\sqrt{|S|}\\
 &\le C\ell\max\{\sqrt q,q\rho\}^\ell\sqrt{|S|}.
 \end{aligned}
\]
The last inequality uses $|S|/n\le(3q/(4\alpha))q^{-\ell}$.
Choose $a$ with $\max\{\sqrt q,q\rho,q-1\}<a<q$.
The factor $\ell$ is absorbed by this larger exponential base,
so $\|f\|_2\le C a^\ell\sqrt{|S|}$.
All constants depend only on $d,\gamma,\alpha$.

For $e=\{u,w\}$, we have $b_e\le f(u)+f(w)$, so
\[
 \sum_e b_e^2\le2d\|f\|_2^2\le C a^{2\ell}|S|.
\]
Take $\varepsilon=q-a$, $L\ge\max\{1,\sqrt C\}$, and
$T=\{e:b_e\ge La^\ell\}$.
Writing $\dist(v,e)$ for the minimum distance from $v$ to an
endpoint of $e$, we obtain
\[
 \sum_{v\in S}|\{e\in T:\dist(v,e)\le\ell-1\}|
 =\sum_{e\in T}b_e
 \le\frac{\sum_e b_e^2}{La^\ell}
 \le La^\ell|S|.
\]
Some $v\in S$ therefore lies within distance $\ell-1$ of at
most $L(q-\varepsilon)^\ell$ edges of $T$, which is property~(B).
\end{proof}

\end{document}